\documentclass{amsart}

\usepackage[english]{babel}
\usepackage{hyperref}
\usepackage{enumitem}
\usepackage{graphicx}
\usepackage{xcolor} 
\usepackage{pgfplots}
\pgfplotsset{compat=1.18}

\newtheorem{theorem}{Theorem}[section]
\newtheorem{definition}[theorem]{Definition}
\newtheorem{proposition}[theorem]{Proposition}
\newtheorem{corollary}[theorem]{Corollary}
\newtheorem{lemma}[theorem]{Lemma}
\newtheorem{example}[theorem]{Example}

\theoremstyle{remark}

\numberwithin{equation}{section}

\begin{document}

\title[Generalized Semi-Riemannian Submersions and Foliations]{Generalized Semi-Riemannian Submersions and Foliations}


\author[B. O. Alves]{Benigno O. Alves}
\address{Benigno O. Alves \textup{(corresponding author)} \hfill\break\indent Instituto de Matem\'{a}tica e Estat\'{\i}stica, Universidade Federal da Bahia \hfill\break\indent Rua Bar\~ao de Jeremoabo, 40170-115 Salvador, Bahia, Brazil}
\email{benignoalves@ufba.br}

\thanks{The authors would like to thank Patricia Marçal and Miguel Angel Javaloyes for their valuable contributions to this work.}

\author[J. C. Oliveira]{Jair C. de Oliveira}
\address{Jair C. de Oliveira \hfill\break\indent Instituto de Matem\'{a}tica e Estat\'{\i}stica, Universidade Federal da Bahia \hfill\break\indent Rua Bar\~ao de Jeremoabo, 40170-115 Salvador, Bahia, Brazil}
\email{jair.carneiro@ufba.br}

\thanks{This study was financed in part by the Coordena\c{c}\~ao de Aperfei\c{c}oamento de Pessoal de N\'{\i}vel Superior - Brasil (CAPES) - Finance Code 001.}

\date{\today}


\begin{abstract}
In this article, we introduce the concepts of generalized semi-Riemannian submersions and foliations, extending the classical framework to accommodate leaves with varying or degenerate causal characters, such as homogeneous foliations on semi-Riemannian manifolds and codimension-one lightlike foliations on Lorentz manifolds (including pp-waves). We establish that a regular foliation with a basic horizontal distribution is generalized semi-Riemannian if and only if it is transnormal. Furthermore, we introduce stationary foliations and prove that a stationary transnormal foliation induces a classical Riemannian foliation on the spatial rest space of a conformal observer. As a geometric consequence, we obtain a rigidity result ruling out stationary codimension-one lightlike foliations on positively curved Robertson-Walker spacetimes or static spacetimes. Finally, we derive an O'Neill-type sectional curvature formula for submersions with an involutive total distribution.
\end{abstract}

\maketitle

\section{Introduction}

Foliation theory provides a powerful framework for studying wave fronts. In particular, Huygens’ principle implies that in a Riemannian (or Finsler) manifold, a wave propagating at constant speed along its fronts is governed by an isoparametric function whose level sets have constant mean curvature.  

When the singular level sets are submanifolds, one obtains a \emph{singular Riemannian foliation}, which on its regular part can be described locally by a Riemannian submersion \cite{wang1987isoparametric}. In anisotropic media, one must instead work on Finsler manifolds \cite{alves2024isoparametric,dehkordi2019huygens,javaloyes2023general,alexandrino2019finsler}. The theory of singular Riemannian foliations is by now well developed \cite{alexandrino2013progress}, whereas singular Finsler foliations have appeared only more recently \cite{alexandrino2019singular,alexandrino2024equifocal}.

In Lorentzian geometry, foliations encode key physical structures: light cones, horizons, and causal propagation all arise as leaves of suitable foliations. If $(M,g)$ is a spacetime, a \emph{spacelike} (resp.\ \emph{timelike}) foliation is one whose leaves have index $0$ (resp.\ index $1$) \cite{walschap1999spacelike,chaves2020foliations,qadir2006foliation,montiel1999uniqueness}. In \cite{caramello2024transverse}, the authors obtained a transverse diameter theorem in the context of Lorentzian foliations, which can be interpreted as a Hawking–Penrose-type singularity theorem for timelike geodesics transverse to the foliation.

When the leaves are degenerate (null), one obtains a \emph{lightlike foliation} \cite{bolosgeometric,duggal2012foliations}. For example, in a pp-wave, we naturally obtain a lightlike foliation; indeed, according to the definitions in \cite{aazami2023finsler,allout2022homogeneous}, such spaces admit a parallel lightlike vector field $N$, and by Corollary~\ref{intppwave}, the distribution $N^{\perp}$ is integrable. We can call this foliation a \emph{pp-wave foliation}. Such foliations lie at the heart of both mathematical relativity and the burgeoning study of gravitational waves. 

In the classical setting, a singular foliation $\mathcal{F}$ is called \emph{transnormal} if its horizontal distribution $\mathcal{H}$ is invariant under the geodesic flow: whenever a geodesic $\gamma$ satisfies $\gamma'(t_{0})\in \mathcal{H}$, it remains horizontal for all time. Transnormal foliations coincide with singular Riemannian (or Finsler) foliations \cite{alexandrino2013progress,alexandrino2019singular}. In Lorentzian geometry, Galloway showed that any null hypersurface admits a geodesically invariant tangent null vector field, hence any codimension-one lightlike foliation is transnormal \cite[Prop.~3.1]{galloway2004null}.

In this work, we bridge Lorentzian geometry and classical Riemannian foliation theory by proving that a stationary $g$-transnormal foliation on a spacetime induces a transnormal foliation on the spatial rest space of a conformal observer (Theorem~\ref{FTRW}). As an important physical and geometric consequence, we establish a topological obstruction for stationary lightlike structures in cosmological backgrounds: whenever the spatial rest space of a Robertson–Walker spacetime, or a static spacetime, has strictly positive sectional curvature and $U$ is the canonical observer field, there exists no non-trivial codimension-one lightlike foliation that is stationary with respect to $U$ (Corollary~\ref{cOrll}). By leveraging Lytchak and Thorbergsson's rigidity theorem for Riemannian foliations on positively curved manifolds, this non-existence result demonstrates that global stationary lightlike wave fronts—such as coherent pp-wave-like configurations—are geometrically prohibited in closed, positively curved FLRW universes.

A particularly rich class of examples is provided by \emph{homogeneous foliations}, i.e.,\ those whose leaves are the orbits of an isometric action. In both Riemannian and Finslerian settings, homogeneous foliations furnish canonical instances of singular transnormal foliations and serve as building blocks for the local and semilocal theory of singular Riemannian foliations \cite{alexandrino2022lie}. Indeed, the slice theorem for group actions implies that, around each point, a homogeneous foliation is determined by the linear isotropy representation, yielding a simple model for the foliation in a tubular neighborhood. Moreover, recent work on semi-Riemannian isometric actions \cite{zeghib1999isometry,chaib2025isometries,allout2022homogeneous} highlights fundamental distinctions from the positive-definite case: orbits may be degenerate, isotropy groups need not be compact, leaves can fail to be closed, and the orbit space may even cease to be Hausdorff. Consequently, homogeneous foliations in semi-Riemannian manifolds not only supply key examples but also guide the development of the general theory by encoding the geometry of group actions into an infinitesimal model for singular transnormal foliations. 

Homogeneous foliations are at the heart of the theory of Very Special Relativity, in which the laws of physics are assumed to be invariant under a proper subgroup of the Lorentz group rather than the full Lorentz group. In this way, the study of the quotient space of isometric actions can provide essential information.

Despite these rich developments, neither the classical semi-Riemannian foliation theory nor its Finsler analogue fully accommodates foliations that feature leaves of mixed causal type or degenerate leaves—most notably homogeneous foliations and codimension-one null foliations such as pp-waves. 

A \emph{(classical) semi-Riemannian foliation} is a transnormal foliation such that the leaves are non-degenerate. A complementary viewpoint comes from submersions. A \emph{semi-Riemannian submersion} 
$\pi\colon (M,g)\to (B,h)$
is a surjective submersion whose fibers are non-degenerate submanifolds and which preserves the metric on horizontal vectors \cite{o1966fundamental,o1983semi}. Semi-Riemannian foliations are precisely those regular foliations locally arising from such submersions \cite{dolgonosova2018pseudo}. The celebrated O'Neill formula then relates the sectional curvatures of $M$ and $B$:
\[
K^h_{\pi(p)}(X,Y)
\;=\;
K^g_{p}(\tilde X,\tilde Y)
\;+\;\frac{3}{4}\,
g_p\bigl([\tilde X,\tilde Y]^{\mathcal V},\,[\tilde X,\tilde Y]^{\mathcal V}\bigr),
\]
where $\tilde X,\tilde Y$ are horizontal lifts and $[\tilde X,\tilde Y]^{\mathcal V}$ is the vertical part of their commutator. In the pseudo-Finslerian setting, analogous formulas have been developed by Javaloyes and Huber \cite{huber2023fundamental}.

Our primary goal is to introduce a unified notion of \emph{generalized semi-Riemannian submersions} and \emph{foliations} which allows leaves of varying causal character, including degenerate ones, within a single framework. A \emph{generalized semi-Riemannian foliation} is a regular transnormal foliation whose horizontal distribution is locally spanned by basic vector fields. This second condition replaces the classical requirement that leaves be non-degenerate, thus allowing the foliation to include leaves with different causal characters, or even lightlike leaves. In Theorem~\ref{transnormalidadeSubmersao}, we show that a regular foliation is generalized semi-Riemannian if and only if it is locally defined by a generalized semi-Riemannian submersion, itself a natural generalization of classical semi-Riemannian submersions \cite{dolgonosova2018pseudo}. Fundamental examples include the regular orbits of isometric actions and foliations by lightlike hypersurfaces.

We investigate the local structure of such foliations, particularly under the assumption that the sum of the vertical and horizontal distributions is involutive, and we derive an O'Neill-type curvature formula (Theorem~\ref{FormuladeOneill}).

\medskip

\section{Distributions, Foliations, and Transnormal Foliations}

To begin, we offer a brief physical motivation. A \emph{wave} on a Riemannian manifold 
$(S,h)$ is a smooth function
$\varphi :
\mathbb R \times S\to \mathbb R$ 
satisfying the classical \emph{wave equation}
$\Delta\varphi_t=\frac{d^2}{dt^2}\varphi_p$
where 
$\varphi_t(\cdot)=\varphi(t,\cdot)$ and $\varphi_p(\cdot)=\varphi(\cdot,p)$. At each fixed time $t$, the \emph{wave front} is given by the level set of $\varphi_t$, whose tangent spaces are defined by $\operatorname{Ker}d\varphi$ for regular points. The Riemannian gradient $\nabla \varphi_t$ indicates the propagation velocity of the wave front. Thus, the collection of wave fronts at time $t$ forms a foliation of $S$.

A \emph{singular distribution} on a manifold $M$ is a map $\mathcal{D}$ that associates with each point $p\in M$ a vector subspace $\mathcal{D}_p$ of $T_pM$. We say that $\mathcal{D}$ is \emph{regular} if $\dim \mathcal{D}_x$ is constant. A \emph{section} of $\mathcal{D}$ is a map $X:M\to \mathcal{D}$ such that $\pi\circ X=\operatorname{Id}|_M$, or equivalently, it is a section of $TM$ tangent to $\mathcal{D}$, where $\pi:TM\to M$ is the canonical projection. A section $X$ of $\mathcal{D}$ is \emph{smooth} if $X(f)\in C^{\infty}(M)$ for each $f\in C^{\infty}(M)$. This is equivalent to saying that $X$ is a smooth vector field on $M$ tangent to $\mathcal{D}$. The set of smooth sections of $\mathcal{D}$ is denoted by $\mathfrak{X}(\mathcal{D})$. The distribution $\mathcal{D}$ is \emph{smooth} if for each $p\in M$ and $v\in \mathcal{D}_p$ there is $X\in \mathfrak{X}(\mathcal{D})$ such that $X(p)=v$. Furthermore, $\mathcal{D}$ is \emph{locally finitely generated} if for any $p\in M$ there are an open neighborhood $U$ of $p$ and vector fields $Y_1,\dots,Y_l\in \mathfrak{X}(U)$ such that $\mathcal{D}|_U=\operatorname{span}\{Y_1,\dots,Y_l\}$.

A \emph{singular foliation} is a partition $\mathcal{F} = \{L_p\}_{p \in M}$ of a manifold $M$ into immersed submanifolds $L_p$ called the \emph{leaf passing through $p$}, where the \emph{vertical distribution} $\mathcal{V}^{\mathcal{F}}=\{\mathcal{V}^{\mathcal{F}}_p:=T_pL_p\}$ is smooth (we will write simply $\mathcal{V}$ when there is no risk of confusion). A \emph{regular foliation} (or simply a \emph{foliation}) is a smooth singular foliation in which all leaves have the same dimension. It is well known that a partition of $M$ into submanifolds of equal dimension $k$ constitutes a foliation if and only if it can be locally described by a submersion. 

A distribution $\mathcal{D}$ on a manifold $B$ is said to be \emph{integrable} if there exists a foliation $\mathcal{F}_{\mathcal{D}} = \{L_x\}_{x \in B}$ such that $T_xL_x = \mathcal{D}_x$ for each $x \in B$. The celebrated \emph{Frobenius Theorem} ensures that a regular distribution is integrable if and only if it is involutive, that is, $[X,Y]\in \mathfrak{X}(\mathcal{D})$ for each $X,Y\in \mathfrak{X}(\mathcal{D})$. 

In the singular case, however, this equivalence fails. Nonetheless, there is a powerful generalization known as the \emph{Stefan-Sussmann Theorem}. 

\begin{theorem}[Stefan-Sussmann, \cite{agrachev2013control}]\label{Tstefansussman}
A singular smooth distribution $\mathcal{D}$ is integrable if and only if $\mathcal{D}$ is involutive and locally finitely generated.
\end{theorem}

Given a semi-Riemannian metric $g$ on $M$, we can derive the following distributions associated with the foliation $\mathcal{F}$: 
\begin{itemize} 
    \item The \emph{horizontal distribution}: $\mathcal{H}:=\mathcal{V}^{\perp}$, which is the orthogonal complement of the vertical distribution at each point; 
    \item the \emph{radical distribution}: $\mathcal{R}=\mathcal{V}\cap \mathcal{H}$; and 
    \item the \emph{total distribution}: $\mathcal{A}=\mathcal{V}+\mathcal{H}$, which combines the vertical and horizontal distributions. 
\end{itemize}
A leaf $L_p$ is said to be \emph{non-degenerate} when $\mathcal{R}_q=\{0\}$ for all $q\in L_p$, or equivalently, if $\mathcal{A}_q=T_qM$. Otherwise, the leaf $L_p$ is considered \emph{degenerate}.

\begin{lemma}\label{SuavidadeDistribuicaoHorizontal}
The horizontal distribution of a regular foliation is a smooth regular $(n-k)$-dimensional distribution.
\end{lemma}
\begin{proof}
Given $p \in M$, consider an open neighborhood $U$ of $p$ and smooth vertical vector fields $V_1, \dots, V_k$ on $U$ such that $\{V_1(q), \dots, V_k(q)\}$ is a basis for $\mathcal{V}_q$ for all $q \in U$. Define the map $F: TM|_U \to \mathcal{V}|_U$ by $F(v) = \sum_{i=1}^k g(v, V_i) V_i$. Note that $\ker F = \mathcal{H}|_U$ and that $F$ is a smooth bundle homomorphism over $U$. Since $\mathcal{H}$ is a regular distribution of dimension $n-k$, it follows that $\mathcal{H}$ is smooth \cite[Theorem 10.34]{lee2013smooth}.
\end{proof}

A foliation $\mathcal{F}$ on a semi-Riemannian manifold is a \emph{lightlike foliation} if its leaves are degenerate submanifolds. If $\mathcal{F}$ is a codimension-one lightlike foliation, there exists a lightlike vector field $N$ such that $N^{\perp}=\mathcal{V}^{\mathcal{F}}$. In this case, $N$ is a \emph{generator vector field} of $\mathcal{F}$. Note that any lightlike foliation on a spacetime is transversal to any restspace foliation.  

\begin{example}\label{intppwave}
Let $(M,g)$ be a spacetime and $N$ be a parallel lightlike vector field. Then the distribution $N^{\perp}$ is integrable. Indeed, given $X,Y \in \mathfrak{X}(N^{\perp})$, metric compatibility and the condition $\nabla N = 0$ yield
\begin{align*}
    g([X,Y],N) = Xg(Y,N) - g(Y,\nabla_X N) - Yg(X,N) + g(X,\nabla_Y N) = 0.
\end{align*}
This implies that $[X,Y] \in \mathfrak{X}(N^{\perp})$, and thus integrability follows from Frobenius' theorem. In particular, every pp-wave spacetime $(M,g,N)$ naturally admits a lightlike foliation defined by $N^{\perp}$. 
\end{example}

An \emph{observer vector field} on a spacetime is a future-directed normalized timelike vector field. A \emph{restspace} of $U$ is an integral manifold of the orthogonal distribution $U^{\perp}$, which is a spacelike semi-Riemannian submanifold. When $U^{\perp}$ is integrable, we say that $U$ is \emph{irrotational} (cf.\ \cite{o1983semi}) and the foliation $\mathcal{F}_{U^{\perp}}=\{S_p\}$ is the \emph{restspace foliation}, where $S_p$ is the restspace containing $p$. 

In \cite[Proposition 6.1]{bolosgeometric}, an integrability criterion was established for a codimension-one lightlike distribution in a spacetime. In the sequel, we present a more general integrability criterion.

\begin{proposition}\label{criteriointegrabilidade}
Let $(M,g)$ be a spacetime and $U$ be an irrotational observer vector field. Let $\mathcal{D} = N^{\perp}$ be a codimension-one null distribution on a spacetime $(M,g)$ with null generator $N$. Then $\mathcal{D}$ is integrable if and only if $\nabla_{N}N =\lambda N$ for some function $\lambda$ and the spacelike distribution $\mathcal{D}\cap U^{\perp}$ is integrable.
\end{proposition}
\begin{proof}
Since $U^{\perp}$ is integrable, for any $X, Y \in \mathfrak{X}(\mathcal{D}\cap U^{\perp})$ we already have $[X,Y]\in \mathfrak{X}(U^{\perp})$. Hence $\mathcal{D}$ is integrable precisely when the spacelike slice $\mathcal{D}\cap U^{\perp}$ is integrable and, for every $X\in \mathfrak{X}(\mathcal{D})$, the commutator $[X,N]$ remains in $\mathfrak{X}(\mathcal{D})$.
We claim that the second condition is equivalent to requiring $\nabla_{N}N\in \mathfrak{X}(\mathcal{R})$, where $\mathcal{R}=\mathcal{D}\cap \mathcal{D}^{\perp}=\operatorname{span}(N)$. Indeed, for any $X\in \mathfrak{X}(\mathcal{D})$, one has
\[
  [X,N]\in \mathfrak{X}(\mathcal{D}) \iff g \bigl([X,N],\,N\bigr)=0.
\]
By metric compatibility of the Levi-Civita connection,
\begin{align*}
    g([X, N], N) &= g(\nabla_X N - \nabla_N X, N) \\
    &= \frac{1}{2} X\big(g(N, N)\big) - N\big(g(X, N)\big) + g(X, \nabla_N N) \\
    &= g(X, \nabla_N N).
\end{align*}
Therefore, $g([X,N],N)=0$ for all $X\in \mathfrak{X}(\mathcal{D})$ if and only if $\nabla_{N}N$ is orthogonal to every $X\in\mathfrak{X}(\mathcal{D})$, i.e., lies in the radical $\mathcal{R}=\operatorname{span} N$ (Lemma~\ref{Pontochave}). This completes the proof. 
\end{proof}

If $\pi:M\to B$ is a submersion, then the regular distribution $\mathcal{V}_p = \ker d\pi_p$ is integrable, and its leaves coincide with the fibers of the submersion. Such a foliation is called a \emph{simple foliation} and is denoted by $\mathcal{F}_{\pi}$. A partition of $M$ is a regular foliation of dimension $k$ if and only if it is locally a simple foliation. A vector field $X$ is called a \emph{projectable field} with respect to a submersion $\pi:M\to B$ if there exists a vector field $Y$ on $B$ such that $d\pi(X_p) = Y_{\pi(p)}$ for every $p \in M$. We say that $X$ is a \emph{foliated vector field} of a foliation $\mathcal{F}$ if its flow maps leaves to leaves. When the foliation is simple, these two definitions coincide.

\begin{lemma}\label{campoprojetavel}
Let $\pi:M\to B$ be a submersion. The following are equivalent:
\begin{enumerate}
    \item $X$ is projectable;
    \item $X$ is a foliated vector field of the foliation $\mathcal{F}_{\pi} = \{\pi^{-1}(x)\}_{x\in B}$; and
    \item $[X,Z]\in \mathfrak{X}(\mathcal{V})$ for every $Z \in \mathfrak{X}(\mathcal{V})$, where $\mathcal{V}=\operatorname{ker}d\pi$ is the vertical distribution.
\end{enumerate}
\end{lemma}

Let $g$ be a semi-Riemannian metric on $M$. A \emph{basic vector field} of a foliation $\mathcal{F}$ on $M$ is a horizontal and foliated vector field. The set of basic vector fields of $\pi$ is denoted by $\mathfrak{X}_b(M)$.

\begin{lemma}\label{lemma1}
Let $(M,g)$ be a semi-Riemannian manifold and let $\pi:M \to B$ be a surjective submersion. If the horizontal distribution $\mathcal{H}$ is locally generated by basic vector fields, then there exists a smooth singular distribution $\mathcal{D}$, locally finitely generated on $B$, such that $\mathcal{D}_{\pi(p)} = d\pi(\mathcal{H}_p) = d\pi(\mathcal{A}_p)$ for every $p \in M$.
\end{lemma}
\begin{proof}
Let $\{\tilde X_1, \dots, \tilde X_l\}$ be basic vector fields that locally generate $\mathcal{H}$. Define $X_i$ by $d\pi \circ \tilde X_i = X_i \circ \pi$. Then the set $\{X_i\}$ generates a smooth singular distribution $\mathcal{D}$ satisfying $d\pi(\mathcal{H}) = \mathcal{D} \circ \pi$.
\end{proof}

If the propagation speed $\lVert \nabla \varphi_t\rVert_h$ is constant on each wave front in a Riemannian manifold $(M,h)$, the function $\varphi_t$ is an \emph{$h$-transnormal function}. Consequently, the integral curves of $\nabla \varphi_t$ are pregeodesics orthogonal to the wave fronts. If, additionally, the Laplacian $\Delta \varphi_t$ is constant along each front, the function is said to be an \emph{isoparametric function} \cite{cartan1938familles}.

A singular foliation $\mathcal{F}$ on a semi-Riemannian manifold $(M,g)$ is \emph{$g$-transnormal} when it satisfies the \emph{transnormal condition}: if $\gamma:I\to M$ is a geodesic such that $\gamma'(t_{0})\in \mathcal{H}$ for some $t_0\in I$, then $\gamma'(t)\in \mathcal{H}$ for all $t\in I$. In this case, $\gamma$ is called an \emph{$\mathcal{F}$-horizontal geodesic}. 

\begin{example}\label{folheacaohomogenea}
The orbits of an isometric action on a semi-Riemannian manifold constitute a singular transnormal foliation, referred to as a \emph{homogeneous foliation}. The proof of this fact is completely analogous to the Riemannian case; see \cite{alexandrino2015lie}.
\end{example}

\begin{example}\label{translightlikeone} 
In Lorentzian geometry, Galloway showed that any null hypersurface admits a geodesically invariant tangent null vector field \cite[Prop.~3.1]{galloway2004null}; hence, any codimension-one lightlike foliation is transnormal.
\end{example}

\begin{example}\label{ExampleWaveFront}
Let $\varphi$ be a wave that propagates in a Riemannian manifold $(S,h)$ with non-zero constant speed along each wave front, meaning there exists a function $b : I \to \mathbb{R}$ such that $h(\nabla \varphi_t, \nabla \varphi_t) = (b \circ \varphi_t)^2$. Now consider the spacetime $M = (\mathbb{R} \times S, g)$ where $g = -dt^2 + h$. Note that $U = \partial_t$ is an irrotational observer vector field and $U^{\perp}(t,p) = \{0\} \times T_p S$. Define $N = (b \circ \varphi_t)U + \nabla \varphi_t$, which is a lightlike vector field. We verify that the distribution $\mathcal{D} = N^{\perp}$ is integrable on $M_r = \{(t,p) \in M : \nabla \varphi_t(p) \neq 0\}$, consequently inducing a codimension-one lightlike transnormal foliation.

First, note that $\mathcal{D} \cap U^{\perp} = \operatorname{Ker} d\varphi_t$, as these spaces have the same dimension, and if $u \in \mathcal{D} \cap U^{\perp}$, then
\begin{align*}
    0 = g(N, u) = g((b \circ \varphi_t)U + \nabla \varphi_t, u) = h(\nabla \varphi_t, u) = d\varphi_t(u),
\end{align*}
implying that $u \in \operatorname{Ker} d\varphi_t$. In particular, $\mathcal{D} \cap U^{\perp}$ is integrable on $M_r$. Thus, by Proposition~\ref{criteriointegrabilidade}, it suffices to prove that $\nabla_N N = \lambda N$ for some function $\lambda$ on $M_r$. 

For a fixed $t$, the function $\varphi_t$ is $h$-transnormal, which implies that the integral curves of $\nabla \varphi_t$ are $h$-pregeodesics; hence, there exists a function $\rho_t$ on $S$ such that
\begin{equation*}
    \nabla^h_{\nabla \varphi_t} \nabla \varphi_t = \rho_t \nabla \varphi_t.
\end{equation*}
On the other hand, one verifies that $g(\nabla_N N, N) = 0$, so in particular $\nabla_N N \in \mathfrak{X}(\mathcal{D})$. To complete the proof, it remains to show that $g(\nabla_N N, X) = 0$ for every $X \in \mathfrak{X}(\mathcal{D} \cap U^{\perp})$. Observe that $X = (0, \hat{X}) \in \mathfrak{X}(\mathcal{D})$ with $\hat{X} \in \operatorname{Ker}(d\varphi_t)$. Since $\nabla_N N = \theta U + \nabla^h_{\nabla \varphi_t} \nabla \varphi_t$ for some function $\theta$, we have
\begin{equation*}
    g(\nabla_N N, X) = h(\nabla^h_{\nabla \varphi_t} \nabla \varphi_t, \hat{X}) = \rho_t h(\nabla \varphi_t, \hat{X}) = \rho_t d\varphi_t(\hat{X}) = 0.
\end{equation*}
\end{example}

Let $\nabla$ be an affine connection on $M$ and $\mathcal{D}$ a smooth regular distribution. We say that $\mathcal{D}$ is \emph{geodesically $\nabla$-invariant} if it is preserved by the flow of the geodesic spray of $\nabla$. More precisely,  
\[
\varphi^G_t(\mathcal{D}_p) = \mathcal{D}_{\pi \circ \varphi^G_t(p)},
\]  
where $G$ is the geodesic spray of $\nabla$ and $\varphi^G$ is its flow, or equivalently, if $\gamma$ is a geodesic with $\gamma'(0) \in \mathcal{D}$, then $\gamma'(t) \in \mathcal{D}$ for any $t$. Lewis' criterion provides a practical way to verify this property. 

\begin{theorem}[Lewis' Criterion, \cite{lewis1998affine}]\label{lewis}
Let $\nabla$ be an affine connection on $TM$, and let $\mathcal{D}$ be a smooth regular distribution. Then, $\mathcal{D}$ is geodesically $\nabla$-invariant if and only if  
\[
\nabla_X Y + \nabla_Y X \in \mathfrak{X}(\mathcal{D})
\]
for all $X, Y \in \mathfrak{X}(\mathcal{D})$.  
\end{theorem}  

\begin{corollary}\label{corLewis}  
Let $\mathcal{F}$ be a smooth regular foliation on a semi-Riemannian manifold $(M,g)$. Then, $\mathcal{F}$ is $g$-transnormal if and only if the horizontal distribution $\mathcal{H}$ satisfies Lewis' criterion.
\end{corollary}  
\begin{proof}  
Note that $\mathcal{F}$ is $g$-transnormal if and only if the horizontal distribution $\mathcal{H}$ is geodesically $\nabla$-invariant, where $\nabla$ is the Levi-Civita connection associated with $g$. Thus, the result follows directly from Theorem~\ref{lewis}.
\end{proof}

\begin{proposition}\label{tranHLieD}
Let $\mathcal{F}$ be a regular foliation. Then, $\mathcal{F}$ is $g$-transnormal if and only if  
\[
\mathcal{L}_U g(X,Y) = 0
\]
for all $X, Y \in \mathfrak{X}(\mathcal{H})$ and $U \in \mathfrak{X}(\mathcal{V})$.  
\end{proposition}  
\begin{proof}  
Computing directly, we obtain  
\begin{align*}  
    \mathcal{L}_U g(X,Y) &= U g(X,Y) - g([U,X],Y) - g([U,Y],X) \\  
    &= g(\nabla_U X, Y) + g(X, \nabla_U Y) - g(\nabla_U X, Y) \\  
    &\quad + g(\nabla_X U, Y) - g(\nabla_U Y, X) + g(\nabla_Y U, X) \\  
    &= g(\nabla_X U, Y) + g(\nabla_Y U, X) \\  
    &= X g(U,Y) - g(\nabla_X Y, U) + Y g(X,U) - g(\nabla_Y X, U) \\  
    &= - g(\nabla_X Y + \nabla_Y X, U).  
\end{align*}  
Thus, the result follows from Corollary~\ref{corLewis}.  
\end{proof}  

\section{Stationary Transnormal Foliations in a Spacetime} 

Let $U$ be an observer in a spacetime $(M,g)$. A codimension-one lightlike distribution $\mathcal{D}$ on $M$ coincides with the distribution orthogonal to a lightlike vector field $N = U + S$, where $S$ is a unit spacelike vector field orthogonal to $U$. If $U$ is irrotational (or synchronizable), then $\mathcal{D}$ is integrable if and only if $\mathcal{D} \cap U^{\perp}$ is also integrable, according to \cite[Theorem 7]{bolosgeometric}. In this case, the foliation generated by $\mathcal{D} \cap U^{\perp}$ is interpreted as the wave fronts observed by $U$, and $S$ is its propagation velocity vector field. 

A wave $\varphi$ is called \emph{stationary} if its wave fronts do not move over time. This motivates the next definition. Let $(M,g)$ be a semi-Riemannian manifold and let $U$ be a nowhere vanishing vector field. A foliation $\mathcal{F} = \{L_p\}$ is said to be \emph{$U$-stationary} if the distribution 
\[
\mathcal{V}^{\mathcal{F}}\cap U^{\perp}=\{T_pL\cap U^\perp(p)\}_{p\in M}
\]
is invariant under the flow of $U$. If $\mathcal{V}^{\mathcal{F}}\cap U^{\perp}$ is integrable, $U$ is a foliated vector field of its integral foliation.  

\begin{proposition}
Let $U$ be an irrotational observer field in a spacetime, and $\mathcal{F}$ be a regular foliation such that $\mathcal{V}^{\mathcal{F}}\cap U^{\perp}$ is an integrable regular distribution. Then $\mathcal{F}$ is $U$-stationary if and only if the distribution
\[
(\mathcal{V}^{\mathcal{F}}\cap U^{\perp}) \;\oplus\; \mathrm{span}\{U\}
\]
is integrable. 
\end{proposition}
\begin{proof}  
Since $U$ is a unit timelike vector field, we have $g(\nabla_Z U, U) = 0$ for any $Z \in \mathfrak{X}(U^\perp)$. Hence
\[
g([Z, U], U)
= g(\nabla_Z U, U) - g(\nabla_U Z, U)
= 2\,g(\nabla_Z U, U)
= 0,
\]
showing that $[U,Z] \in \mathfrak{X}(U^\perp)$. Now, because $\mathcal{V}^{\mathcal{F}}\cap U^{\perp}$ is Lie bracket closed, the distribution $(\mathcal{V}^{\mathcal{F}}\cap U^{\perp}) \oplus \mathrm{span}\{U\}$ is integrable if and only if $[U,Z] \in \mathfrak{X}(\mathcal{V}^{\mathcal{F}}\cap U^{\perp})$ for every $Z \in \mathfrak{X}(\mathcal{V}^{\mathcal{F}}\cap U^{\perp})$. The desired conclusion now follows immediately from Lemma~\ref{campoprojetavel}.
\end{proof}

In \cite[Theorem 8]{bolosgeometric}, Bolos showed that a codimension-one lightlike foliation $\mathcal{F}_{\mathcal{D}}$ generated by the lightlike vector field $N = U + X$ is $U$-stationary if and only if the opposite lightlike distribution $\mathcal{D}^- = (U - X)^\perp$ is integrable. Here, the foliations $\mathcal{F}_{\mathcal{D}}$ and $\mathcal{F}_{\mathcal{D}^-}$ propagate in opposite null directions $X$ and $-X$, respectively, which is a hallmark of a stationary wave. The following result demonstrates that stationary foliations inherently preserve this geometric behavior, as an irrotational observer perceives a transnormal foliation on the spatial rest space.

\begin{theorem}\label{FTRW}  
Let $(M, g)$ be a semi-Riemannian manifold and let $U \in \mathfrak{X}(M)$ be a vector field such that $g(U,U)$ is nowhere zero and the flow of $U$ consists of conformal maps. Let $\mathcal{F}$ be a $U$-stationary, $g$-transnormal foliation on $M$, and let $S \subset M$ be a hypersurface such that $U(q) \in (T_q S)^\perp$ for each $q \in S$. If $\mathcal{F}_S := \mathcal{F} \cap S$ is a regular foliation, then $\mathcal{F}_S$ is an $h$-transnormal foliation on $(S, h)$, where $h = g|_S$.
\end{theorem} 
\begin{proof}
Let $\mathcal{H}^S$ and $\mathcal{V}^S$ denote the horizontal and vertical distributions of $\mathcal{F}_S$, respectively. Given $p \in S$ and $X, Y \in \mathfrak{X}(\mathcal{H}^S)$, our goal is to show that $(\mathcal{L}_V h)_p(X, Y) = 0$ for any $V \in \mathfrak{X}(\mathcal{V}^S)$.

First, assume that $U(p) \notin (T_p L_p)^\perp$, where $L_p$ denotes the leaf of $\mathcal{F}$ passing through $p$. By continuity, there exists an open neighborhood $U_0$ of $p$ in $S$ such that $U(q) \notin (T_q L_q)^\perp$ for all $q \in U_0$; in particular, $\mathcal{F}$ is transverse to $U_0$. Define $W = \{\varphi_t(q) : t \in (-\epsilon, \epsilon) \text{ and } q \in U_0\}$, where $\varphi$ is the flow of $U$. By shrinking $U_0$ or $\epsilon > 0$ if necessary, there exists a vertical vector field $Z \in \mathfrak{X}(\mathcal{V}|_W)$ such that $g(Z, U) \neq 0$ on $W$.
We define vector fields $\widetilde{X}, \widetilde{Y}$ on $W$ via push-forward along the flow $\varphi_t$ of $U$:
\[
\widetilde{X} \circ \varphi_t = d\varphi_t(X), \quad \widetilde{Y} \circ \varphi_t = d\varphi_t(Y).
\]
Let $\alpha, \beta \in C^\infty(W)$ be defined by
\[
\alpha := -\frac{g(\widetilde{X}, Z)}{g(U, Z)} \quad \text{and} \quad \beta := -\frac{g(\widetilde{Y}, Z)}{g(U, Z)}.
\]
Since $U(q) \notin (T_q L_q)^\perp$ for each $q \in U_0$, these functions are well-defined. Construct the corrected vector fields $\widehat{X} = \widetilde{X} + \alpha U$ and $\widehat{Y} = \widetilde{Y} + \beta U$. By construction, $g(\widehat{X}, Z) = 0$. For any $V \in \mathfrak{X}(\mathcal{V}^S)$, let $\widetilde{V}$ be its extension to $W$ satisfying $[\widetilde{V}, U] = 0$. The condition $d\varphi_t(\mathcal{V}^S) \subset \mathcal{V}^{\mathcal{F}}$ guarantees that $\widetilde{V}$ is vertical for $\mathcal{F}$. Since the flow of $U$ is conformal (with factor $\sigma \in C^\infty(W)$), we have
\[
g(\widehat{X}, \widetilde{V}) = g(d\varphi_t X, d\varphi_t V) + \alpha \sigma g(U, V) = \sigma g(X, V) + 0 = 0,
\]
which proves that $\widehat{X}$ and $\widehat{Y}$ are horizontal vector fields with respect to $\mathcal{F}$. Observing that $g(\widetilde{X}, U) = 0$ and $g(\widetilde{Y}, U) = 0$ on $S$ (and by extension on $W$, since $U \perp S$ and the flow is conformal), it follows that $g(\widehat{X}, \widehat{Y}) = g(\widetilde{X}, \widetilde{Y}) + \alpha \beta g(U, U)$. Since $[\widetilde{V}, U] = 0$, the Lie derivative satisfies:
\begin{align*}
\mathcal{L}_{\widetilde{V}} g(\widehat{X}, \widehat{Y}) &= \widetilde{V}\big(g(\widetilde{X}, \widetilde{Y}) + \alpha\beta g(U,U)\big) - g([\widetilde{V}, \widehat{X}], \widehat{Y}) - g([\widetilde{V}, \widehat{Y}], \widehat{X}) \\
&= \mathcal{L}_{\widetilde{V}} g(\widetilde{X}, \widetilde{Y}) + \widetilde{V}(\alpha\beta)g(U,U) - \left( \widetilde{V}(\alpha)\beta g(U,U) + \widetilde{V}(\beta)\alpha g(U,U) \right) \\
&= \mathcal{L}_{\widetilde{V}} g(\widetilde{X}, \widetilde{Y}).
\end{align*}
The $g$-transnormality of $\mathcal{F}$ ensures that $\mathcal{L}_{\widetilde{V}} g(\widehat{X}, \widehat{Y}) = 0$. Restricting back to $S$ (at $t=0$), we conclude that $(\mathcal{L}_V h)_p(X, Y) = 0$.

Now, suppose that $U(p) \in (T_p L_p)^\perp$. If there exists a sequence of points $\{p_n\} \subset S$ converging to $p$ such that $U(p_n) \notin (T_{p_n} L_{p_n})^\perp$ for all $n \in \mathbb{N}$, then by the previous case, $\mathcal{F}_S$ is $h$-transnormal in an open neighborhood of each $p_n$. This implies $(\mathcal{L}_V h)_{p_n}(X,Y) = 0$ for all $n$, and therefore $(\mathcal{L}_V h)_p(X,Y) = 0$ by continuity.

Otherwise, if no such sequence exists, there exists an open neighborhood $U_0$ of $p$ in $S$ such that $U(q) \in (T_q L_q)^\perp$ for all $q \in U_0$. In this case, $\mathcal{V}^S|_{U_0}$ coincides with $\mathcal{V}|_{U_0}$, meaning that the leaves of $\mathcal{F}_S|_{U_0}$ coincide with open subsets of the leaves of $\mathcal{F}$. Consequently, $V$ is the restriction of a vertical vector field $\widetilde{V}$ of $\mathcal{F}$, while $X, Y$ are restrictions of horizontal vector fields $\widetilde{X}, \widetilde{Y}$ of $\mathcal{F}$. Thus, the $g$-transnormality of $\mathcal{F}$ implies \( (\mathcal{L}_V h)_p(X, Y) = (\mathcal{L}_{\widetilde{V}} g)_p(\widetilde{X}, \widetilde{Y}) = 0.
\)
This completes the proof.
\end{proof}

\begin{corollary}\label{cOrll}
Let $U$ be a conformal, irrotational observer vector field on a spacetime $(M,g)$, let $S$ be a restspace of $U$, and let $\mathcal{F}$ be a $U$-stationary, codimension-one lightlike foliation on $M$. Then $\mathcal{F}_S := \mathcal{F} \cap S$ is a codimension-one Riemannian foliation on the Riemannian manifold $(S, h)$, where $h = g|_S$. Furthermore, if $S$ is a closed manifold with strictly positive sectional curvature, then there exists no $U$-stationary, codimension-one lightlike foliation on $M$.
\end{corollary}
\begin{proof}
A codimension-one lightlike foliation $\mathcal{F}$ is $g$-transnormal. Since $U$ is timelike and $\mathcal{F}$ is lightlike, $U$ is nowhere tangent to $\mathcal{F}$, and thus $U$ is transversal to $\mathcal{F}$. Since $U$ is irrotational and conformal, Theorem~\ref{FTRW} implies that $\mathcal{F}_S$ is an $h$-transnormal foliation on $S$, which corresponds to a codimension-one Riemannian foliation on $(S, h)$. The second assertion then follows directly from Lytchak and Thorbergsson's rigidity result (\cite[Theorem 1.1]{lytchak2016riemannian}), which rules out the existence of codimension-one Riemannian foliations on positively curved closed manifolds.
\end{proof}

\section{Generalized Semi-Riemannian Foliations}

It is not true that a regular homogeneous foliation is a classical semi-Riemannian foliation, since the leaves can have distinct causal characters and even be degenerate. Let us establish a generalization that includes homogeneous foliations.

\begin{definition}
A \emph{generalized semi-Riemannian metric} on a manifold $B$ is a pair $(\mathcal{D},h)$ where $\mathcal{D}$ is a smooth singular distribution and $h$ is a smooth assignment that associates to each $p\in M$ a scalar product $h_p$ on $\mathcal{D}_p$. In this case, $(B,\mathcal{D},h)$ is called a \emph{generalized semi-Riemannian manifold}. 
\end{definition}

\begin{example}
Let $h = -\mathbf{h} dt \otimes dt + \mathbf{h}^{-1} dr \otimes dr + r^2 d\sigma^2$ be the Schwarzschild metric on $B = \mathbb{R} \times (0,\infty) \times S^2$, where $\mathbf{h}(r) = 1 - \frac{2m}{r}$, $m > 0$, and $S^2$ is the two-dimensional unit sphere \cite{o1983semi}. Such a metric models a universe with a single star or black hole. Consider the singular distribution $\mathcal{D}$ given for $p = (t,r,x) \in B$ by $\mathcal{D}_p = T_p S$ if $r = 2m$, where $S = (t,2m,S^2)$, and $\mathcal{D}_p = T_p B$ otherwise. Note that if $r = 2m$, then $p$ is contained within the event horizon. This distribution is generated around $p = (t,2m,x)$ by $X_1(t,r,x) = \mathbf{h}\partial_t$, $X_2 = \mathbf{h}\partial_r$, and $X_3 = (0,0,\tilde{X}_3)$, $X_4 = (0,0,\tilde{X}_4)$, where $\tilde{X}_3, \tilde{X}_4$ are linearly independent vector fields on $S^2$ defined around $x$. The restriction of $h$ to $\mathcal{D}$ is a generalized semi-Riemannian metric because $h(X_1,X_1) = -\mathbf{h}^3$, $h(X_2,X_2) = \mathbf{h}$, and $h(X_1,X_2) = 0$ are smooth, and $g(X_i,X_j)$ is also smooth for $i,j \ge 2$. 
\end{example}

\begin{definition} 
Let $(M,g)$ be a semi-Riemannian manifold, $(B,\mathcal{D},h)$ be a singular semi-Riemannian manifold, and $\pi:M\to B$ be a surjective submersion. The map $\pi:(M,g)\to (B,\mathcal{D},h)$ is a \emph{generalized semi-Riemannian submersion} if 
\begin{enumerate}
    \item $\mathcal{H}$ is locally generated by basic vector fields and $\mathcal{D}_{\pi} = d\pi(\mathcal{H}) = d\pi(\mathcal{A})$; and 
    \item $g_p(u,v) = h_{\pi(p)}(d\pi(u),d\pi(v))$ for all $p\in M$ and $u,v\in \mathcal{H}_p$. 
\end{enumerate}
Note that in this case $\mathcal{H}$ is locally finitely generated by basic fields, because it is a regular distribution.  

A \emph{generalized semi-Riemannian foliation} on a semi-Riemannian manifold $(M,g)$ is a regular foliation such that for each $p\in M$ there exist an open neighborhood $U\subset M$ of $p$ and a generalized semi-Riemannian submersion $\pi:(U,g|_{TU\times TU})\to (B,\mathcal{D},h)$ such that $\mathcal{F}|_U = \{L_p\cap U\}_{p\in U}$ coincides with the foliation $\mathcal{F}_{\pi}$ given by the fibers of $\pi$. 
\end{definition}

Note that the codimension of $\mathcal{D}_{\pi(p)}$ coincides with the dimension of $\mathcal{R}_p$. In particular, a fiber $L = \pi^{-1}(x)$ is non-degenerate if and only if $\mathcal{D}_x = T_x B$. Moreover, for any subspace $T \subset \mathcal{H}_p$ such that $\mathcal{A}_p = T \oplus \mathcal{V}_p$, the restriction of $d\pi_p$ to $T$,
\[
d\pi_p|_T : (T, g) \to (\mathcal{D}_{\pi(p)}, h),
\]
is a linear isometry. These properties, combined with Condition 1 of the preceding definition, yield the following result.

\begin{lemma}
If $\pi:(M,g)\to (B,\mathcal{D},h)$ is a generalized semi-Riemannian submersion, then $\dim \mathcal{R}$ and $\operatorname{Ind}(\mathcal{V})$ are constant along the fibers.  
\end{lemma}

\begin{example}[Classical Semi-Riemannian Submersion] 
Let $(M,g)$ be a semi-Riemannian manifold and consider a submersion $\pi:(M,g)\to B$ whose fibers are semi-Riemannian submanifolds of $(M,g)$, meaning that $g$ is non-degenerate along the fibers. In this case, $d\pi(\mathcal{H}_p) = T_{\pi(p)}B$ for all $p\in M$, and thus the definition of a generalized semi-Riemannian submersion coincides with the classical definition of a semi-Riemannian submersion when $(M,g)$ is a semi-Riemannian manifold.  
\end{example}

\begin{example}
Consider the semi-Riemannian manifold $M = (\mathbb{L}^n \setminus \{0\}) \oplus \mathbb{R}^m \oplus \mathbb{R}^k$ endowed with the product metric $g = \langle \cdot, \cdot \rangle_1 \oplus \langle \cdot, \cdot \rangle \oplus \langle \cdot, \cdot \rangle$, where $\langle \cdot, \cdot \rangle_1$ denotes the Minkowski metric and $\langle \cdot, \cdot \rangle$ denotes the standard Euclidean inner product. Define the map $\pi: M \to \mathbb{R} \oplus \mathbb{R}^m$ by
\[
\pi(x,y,w) = (q(x),y),
\]
where $q(x) = \langle x,x \rangle_1$. This map is a submersion since its differential at $p = (x,y,w) \in M$ is given by
\[
d\pi_{(x,y,w)}(u,v,z) = (2\langle x,u \rangle_1, v).
\]
Now, define the distribution $\mathcal{D}$ on $\mathbb{R} \oplus \mathbb{R}^m$ by
\[
\mathcal{D}_{(0,y)} := \{0\} \times T_y\mathbb{R}^m \quad \text{and} \quad \mathcal{D}_{(r,y)} := T_{(r,y)}(\mathbb{R} \oplus \mathbb{R}^m) \text{ if } r \neq 0.
\]
This distribution is smooth because it is globally generated by the smooth vector fields $X(r,y) = (r, 0)$ and $Y_i(r,y) = (0, e_i)$ for $i = 1, \dots, m$, where $\{e_i\}$ is the canonical basis of $\mathbb{R}^m$. 

Define the singular semi-Riemannian metric $h$ on $\mathcal{D}$ by
\[
h_{(r,y)}((t_1,v_1), (t_2,v_2)) = h^0_r(t_1,t_2) + \langle v_1,v_2 \rangle,
\]
where $h^0_r(t_1,t_2) = \frac{t_1 t_2}{4r}$ for $r \neq 0$, and $h^0_0 = 0$. The tensor $h$ is smooth on $\mathcal{D}$ because its evaluation on the generating vector fields yields smooth functions: $h(X,X) = r/4$, $h(X,Y_i) = 0$, and $h(Y_i,Y_j) = \delta_{ij}$.

To verify the generalized metric condition, note that the horizontal space (the orthogonal complement to the fibers of $\pi$ with respect to $g$) at a point $p = (x,y,w)$ is spanned by $(x, 0, 0)$ and the vectors $(0, e_i, 0)$ for $i = 1, \dots, m$. 
Let $V = (x,v,0)$ be an arbitrary horizontal vector at $p$. We apply the differential $d\pi_p$ to $V$ to get $d\pi_p(V) = (2\langle x,x \rangle_1, v) = (2\pi_1(p), v)$, where $\pi_1(p) = q(x)$.

If $\pi_1(p) \neq 0$, then evaluating $h$ on the pushed-forward horizontal vector yields
\begin{align*}
    h_{\pi(p)}(d\pi_p V, d\pi_p V) &= h^0_{q(x)}(2q(x), 2q(x)) + \langle v,v \rangle \\ 
    &= \frac{4q(x)^2}{4q(x)} + \langle v,v \rangle = \langle x,x \rangle_1 + \langle v,v \rangle = g(V,V).
\end{align*}
On the other hand, if $\pi_1(p) = 0$ (meaning $x$ is a null vector), then $d\pi_p(V) = (0, v) \in \mathcal{D}_{(0,y)}$. In this case, $h^0_0 = 0$, and we obtain
\[
h_{\pi(p)}(d\pi_p V, d\pi_p V) = 0 + \langle v,v \rangle = \langle x,x \rangle_1 + \langle v,v \rangle = g(V,V).
\]
Thus, we conclude that $\pi: (M, g) \to (\mathbb{R} \oplus \mathbb{R}^m, \mathcal{D}, h)$ is a generalized semi-Riemannian submersion.
\end{example}

\begin{example}
On $M = \mathbb{L}^n_1 \setminus \{t e_n : t \leq 0\}$, consider the foliation $\mathcal{F}$ whose leaf $L_p$ through a point $p$ is given by the submanifold
\[
L_p = \{q \in M : \langle q, q \rangle_1 = \langle p, p \rangle_1 \text{ and } \langle q, e_n \rangle_1 < 0\}
\]
if $\langle p, p \rangle_1 < 0$ and $\langle p, e_n \rangle_1 < 0$; otherwise, $L_p$ is the cone
\[
\{q + t e_n : \langle q, q \rangle_1 = 0 \text{ and } \langle q, e_n \rangle_1 < 0\}
\]
containing $p$ for some $t \leq 0$. One can easily verify that this constitutes a generalized semi-Riemannian foliation.
\end{example}

The \emph{horizontal lift} of a vector $v \in \mathcal{D}_q$ at $p \in \pi^{-1}(q)$ is a vector $\tilde{v}$ such that $d\pi(\tilde{v}) = v$. The horizontal lift of a section $Y$ of $\mathcal{D}$ is a section $\tilde{Y}$ of $\mathcal{H}$ such that $d\pi(\tilde{Y}) = Y \circ \pi$.  

\begin{lemma}[Existence of Horizontal Lifts]  
Let $(M,g)$ be a semi-Riemannian manifold, and let $\pi:M \to B$ be a surjective submersion such that $\mathcal{H}$ is locally generated by basic vector fields. If $Y$ is a smooth section of $\mathcal{D} = d\pi(\mathcal{H})$, then there exists a basic vector field $\tilde{Y}$ such that $d\pi(\tilde{Y}) = Y \circ \pi$.  
\end{lemma}  
\begin{proof}  
We prove the result locally. Suppose that $X_1, \dots, X_k$ are basic vector fields that locally generate $\mathcal{H}$. Let $Y_1, \dots, Y_k$ be vector fields $\pi$-related to $X_1, \dots, X_k$, respectively. Note that $\{Y_1, \dots, Y_k\}$ locally generates $\mathcal{D}$. Thus, there exist smooth functions $a_1, \dots, a_k$ such that $Y = \sum a_i Y_i$.  

Consider the functions $\tilde{a}_1, \dots, \tilde{a}_k$ defined by $\tilde{a}_i = a_i \circ \pi$, and define $\tilde{Y} = \sum \tilde{a}_i X_i$. By construction, $\tilde{Y}$ is basic and satisfies $d\pi(\tilde{Y}) = Y \circ \pi$.  
\end{proof}  

\begin{lemma}\label{lemaunilev}
The horizontal lift of vectors in $\mathcal{D}(q)$ at $p \in \pi^{-1}(q)$ is unique if and only if the vertical distribution $\mathcal{V}_p$ is non-degenerate. Moreover, if $\tilde{v}$ and $\hat{v}$ are two horizontal lifts of $v$, then there exists $v_0 \in \mathcal{R}_p$ such that $\tilde{v} = \hat{v} + v_0$. In particular, $g(\tilde{v}, \tilde{v}) = g(\hat{v}, \hat{v})$.  
\end{lemma}  
\begin{proof}  
If $\pi^{-1}(q)$ is non-degenerate, then $T_pM = \mathcal{H}(p) \oplus T_p\pi^{-1}(q)$ and $\mathcal{D}(q) = T_qB$. In particular, the restriction $d\pi_p|_{\mathcal{H}(p)}: \mathcal{H}(p) \to T_qB$ is a linear isomorphism, implying the uniqueness of the horizontal lift.  

To prove the converse, assume that $\pi^{-1}(q)$ is degenerate and show that the horizontal lift is not unique. Given $v \in \mathcal{D}(q)$ and $p \in \pi^{-1}(q)$, let $\tilde{v}$ be a horizontal lift of $v$. Since $\pi^{-1}(q)$ is degenerate, there exists a nontrivial vector $v_0 \in \mathcal{R}_p$, which, in particular, is both horizontal and vertical. Thus, $\tilde{v} + v_0$ is horizontal and satisfies $d\pi(\tilde{v} + v_0) = d\pi(\tilde{v}) = v$, meaning that $\tilde{v} + v_0$ is also a horizontal lift of $v$.  
\end{proof}  

\begin{lemma}\label{Lema1}  
Let $(M,g)$ be a semi-Riemannian manifold, and let $\pi:(M,g) \to B$ be a submersion such that the horizontal distribution is locally generated by basic vector fields. If $U$ is vertical and $\tilde{X}$ is the horizontal lift of a vector field $X \in \mathfrak{X}(\mathcal{D})$, then $[U, \tilde{X}]$ is also vertical. If $\mathcal{D} = d\pi(\mathcal{H})$ is involutive, then $[\tilde{X}, \tilde{Y}] - \widetilde{[X,Y]}$ is vertical for any $X, Y \in \mathfrak{X}(\mathcal{D})$, where $\tilde{X}, \tilde{Y}$, and $\widetilde{[X,Y]}$ are the horizontal lifts of $X$, $Y$, and $[X,Y]$, respectively.  
\end{lemma}  
\begin{proof}  
By Lemma~\ref{campoprojetavel}, $[U, \tilde{X}]$ is vertical. Now, we verify the second statement. Observe that a smooth vector field $U$ on $M$ is vertical if and only if $U(f \circ \pi) = 0$ for any smooth function $f: B \to \mathbb{R}$. Thus,  
\[
(\widetilde{[X,Y]} - [\tilde{X},\tilde{Y}])(f \circ \pi) = (d\pi(\widetilde{[X,Y]}) - d\pi([\tilde{X},\tilde{Y}]))(f) = 0,
\]  
since $d\pi(\widetilde{[X,Y]}) = [X,Y] = d\pi([\tilde{X},\tilde{Y}])$.  
\end{proof}  

\begin{theorem}\label{T} 
Let $(M,g)$ be a semi-Riemannian manifold, and let $\pi: M \to B$ be a surjective submersion such that the horizontal distribution $\mathcal{H}$ is locally generated by basic vector fields. Then, there exists a generalized semi-Riemannian metric $(\mathcal{D},h)$ on $B$ such that $\pi: (M,g) \to (B,\mathcal{D},h)$ is a generalized semi-Riemannian submersion if and only if  
\[
\mathcal{L}_U g(X,Y) = 0
\]
for all vector fields $U \in \mathfrak{X}(\mathcal{V})$ and $X,Y \in \mathfrak{X}(\mathcal{H})$.  
\end{theorem}  
\begin{proof}  
Suppose first that $\pi: (M,g) \to (B,\mathcal{D},h)$ is a generalized semi-Riemannian submersion. For basic vector fields $\tilde{X}, \tilde{Y} \in \mathfrak{X}(\mathcal{H})$, we have  
\[
\mathcal{L}_U g(\tilde{X}, \tilde{Y}) = U g(\tilde{X}, \tilde{Y}) = 0.
\]
Since $\mathcal{L}_U g$ is $C^{\infty}(M)$-bilinear and $\mathcal{H}$ is locally generated by basic vector fields, it follows that  
\[
\mathcal{L}_U g(X,Y) = 0
\]
for all $X, Y \in \mathfrak{X}(\mathcal{H})$.  

Conversely, suppose that $\mathcal{L}_U g(X,Y) = 0$ for all $U \in \mathfrak{X}(\mathcal{V})$ and $X,Y \in \mathfrak{X}(\mathcal{H})$. By hypothesis and Lemma~\ref{lemma1}, there exists a distribution $\mathcal{D}$ on $B$ such that $d\pi(\mathcal{H}_p) = \mathcal{D}_{\pi(p)}$ for all $p \in M$. We define a metric $h$ on $\mathcal{D}$ by  
\[
h_q(x,y) = g_p(\tilde{x},\tilde{y}),
\]
where $p \in \pi^{-1}(q)$ and $\tilde{x}, \tilde{y} \in \mathcal{H}_p$ are the horizontal lifts of $x, y \in \mathcal{D}_q$, respectively.  

To show that $h$ is well-defined, given $X,Y \in \mathfrak{X}(\mathcal{D})$, let $\tilde{X},\tilde{Y} \in \mathfrak{X}(\mathcal{H})$ be their respective horizontal lifts. For any vertical vector field $U \in \mathfrak{X}(\mathcal{V})$, using $g(\mathcal{V}, \mathcal{H}) = 0$ and $[U, \tilde{X}] \in \mathfrak{X}(\mathcal{V})$, we have  
\[
0 = \mathcal{L}_U g(\tilde{X},\tilde{Y}) = U g(\tilde{X},\tilde{Y}) - g([U,\tilde{X}],\tilde{Y}) - g(\tilde{X},[U,\tilde{Y}]) = U g(\tilde{X},\tilde{Y}).
\]
Thus, the function $g(\tilde{X}(p),\tilde{Y}(p))$ is constant along the connected fibers of $\pi$. This, together with Lemma~\ref{lemaunilev}, implies that $h$ is well-defined. Finally, using Lemma~\ref{Pontochave}, we conclude that $h_q$ is a non-degenerate scalar product on $\mathcal{D}_q$, completing the proof.
\end{proof}  

Proposition~\ref{tranHLieD} and Theorem~\ref{T} imply the following result, which generalizes \cite[Theorem 1]{dolgonosova2018pseudo} to foliations that admit degenerate leaves.

\begin{theorem}\label{transnormalidadeSubmersao}
Let $(M,g)$ be a semi-Riemannian manifold and let $\mathcal{F}$ be a foliation such that the horizontal distribution is locally generated by basic vector fields. Then $\mathcal{F}$ is $g$-transnormal if and only if $\mathcal{F}$ is a generalized semi-Riemannian foliation.  
\end{theorem}

\begin{lemma}\label{lemma2coro}  
Let $(M,g)$ be a semi-Riemannian manifold, and let $\mu: G \times M \to M$ be an isometric action such that $M / G$ is a manifold and $\pi: M \to M / G$ is a submersion. Then, the horizontal distribution $\mathcal{H}$ is locally generated by basic vector fields. Consequently, any regular homogeneous foliation is a generalized semi-Riemannian foliation.
\end{lemma}  
\begin{proof}  
Since $\mathcal{H}$ is a smooth regular distribution (Lemma~\ref{SuavidadeDistribuicaoHorizontal}), for each $p \in M$, there exists a family of horizontal vector fields $\{Y_1, \dots, Y_k\}$ defined on an open neighborhood $U_0$ of $p$ that locally generates $\mathcal{H}$. By the Local Form Theorem for Submersions, there exist an open neighborhood $U$ of $p$ in $M$, an open neighborhood $V$ of the identity $e$ in $G$, and a submanifold $S \subset M$ passing through $p$ such that:  
\begin{enumerate}  
    \item $T_q M = T_q S \oplus T_q G(q)$ for every $q \in S$;  
    \item $U = \{ \mu_g(q) \mid q \in S, \ g \in V \}$;  
    \item $\{ \mu_g(q) \mid g \in V \} \cap S = \{ q \}$ for every $q \in S$.  
\end{enumerate}  
In particular, the action map $\mu: V \times S \to U$ is a diffeomorphism onto $U$. For each $i \in \{1, \dots, k\}$, define a vector field $\widetilde{Y}_i$ on $U$ by  
\[
\widetilde{Y}_i(\mu_g(q)) = d(\mu_g)_q(Y_i(q))
\]
for any $\mu_g(q) \in U$. Since $\mu$ acts by isometries, the family $\{\widetilde{Y}_1, \dots, \widetilde{Y}_k\}$ generates $\mathcal{H}$ on $U$. 

To verify this, consider $p' = \mu_g(q) \in U$ with $q \in S$. Given $u \in T_{p'} G(q) = \mathcal{V}_{p'}$, there exists $v \in T_q G(q) = \mathcal{V}_q$ such that $d(\mu_g)_q(v) = u$. Since $\mu_g$ is an isometry,
\[
g_{p'}(\widetilde{Y}_i(p'), u) = g_{p'}\big(d(\mu_g)_q(Y_i(q)), \, d(\mu_g)_q(v)\big) = g_q(Y_i(q), v) = 0,
\]
proving that $\widetilde{Y}_i$ is horizontal. To show that $\widetilde{Y}_i$ is projectable (and hence basic), note that $\pi \circ \mu_g = \pi$ for all $g \in G$. Therefore,
\[
d\pi_{p'}(\widetilde{Y}_i(p')) = d\pi_{\mu_g(q)}\big(d(\mu_g)_q(Y_i(q))\big) = d(\pi \circ \mu_g)_q(Y_i(q)) = d\pi_q(Y_i(q)),
\]
which depends only on $q \in S$ and is independent of $g \in V$. Hence, $\widetilde{Y}_i$ is projectable.
\end{proof}  

\begin{example}
Consider the isometric Riemannian action of $\mathbb{S}^1$ on the round sphere $\mathbb{S}^3$ that generates the Hopf fibration. Let $\mathcal{F}$ be the foliation given by the orbits of this action, and let $X$ be the unit Killing vector field generating its leaves. Consider $M = \mathbb{S}^3 \times \mathbb{R}$ equipped with the pseudo-Riemannian metric $\tilde{g} = g - dt^2$, where $g$ is the round metric on $\mathbb{S}^3$, and consider the vector field $\tilde{X} = (X,1) \in \mathfrak{X}(M)$. Define the submersion $\pi: \mathbb{S}^3 \times \mathbb{R} \to \mathbb{S}^3$ by $\pi(p,t) = \varphi^{\tilde{X}}_{-t}(p,t) \in \mathbb{S}^3 \times \{0\} \cong \mathbb{S}^3$.
The fibers of $\pi$ coincide with the leaves of the foliation $\widetilde{\mathcal{F}}$ generated by $\tilde{X}$, which is a one-dimensional lightlike foliation. Observe that $\widetilde{\mathcal{F}}$ is also a homogeneous foliation, implying that its horizontal distribution $\widetilde{\mathcal{H}}$ is locally generated by basic vector fields. Thus, $\pi: (M, \tilde{g}) \to \big(\mathbb{S}^3, \, \mathcal{D} := d\pi(\widetilde{\mathcal{H}}), \, h := g|_{\mathcal{D} \times \mathcal{D}}\big)$
is a generalized pseudo-Riemannian submersion. Furthermore, one verifies that $\mathcal{D} = d\pi(\widetilde{\mathcal{H}})$ coincides with the horizontal distribution of the classical Hopf fibration, which is non-involutive.
\end{example}

\subsection{O'Neill-Type Operator and Mean Curvature Form for \texorpdfstring{$\mathcal{A}$}{A} Involutive}

Let $\pi: (M,g) \to (B, \mathcal{D}, h)$ be a generalized semi-Riemannian submersion such that the distribution $\mathcal{A}$ is involutive (in particular, $\mathcal{D}$ is involutive). Given $Z, W \in \mathfrak{X}_b(M)$, the vector field  
\[
[Z,W] - \widetilde{[d\pi(Z), d\pi(W)]}
\]  
is vertical (see Lemma~\ref{Lema1}), where $\tilde{X}$ denotes a horizontal lift of a vector field $X$ tangent to $\mathcal{D}$. Along lightlike fibers, the horizontal lift is not unique, which introduces additional subtleties in defining geometric structures on the quotient space. However, any two horizontal lifts of the same vector field differ by an element in the radical. Therefore, we can well-define the function $\mathbf{A}: \mathfrak{X}_b(M) \times \mathfrak{X}_b(M) \to \mathbb{R}$ by  
\[
\mathbf{A}(Z,W) = g\big([Z,W] - \widetilde{[d\pi(Z), d\pi(W)]}, \, [Z,W] - \widetilde{[d\pi(Z), d\pi(W)]}\big).
\]  

Let $\mathcal{S}$ be a smooth spacelike subdistribution of $\mathcal{V}$ such that $\mathcal{A} = \mathcal{S} \oplus \mathcal{H}$. In particular, for each $v \in \mathcal{A}$, there exist unique elements $v^{\mathcal{S}} \in \mathcal{S}$ and $v^{\mathcal{H}} \in \mathcal{H}$ such that $v = v^{\mathcal{S}} + v^{\mathcal{H}}$. Since $\mathcal{A}$ is involutive by hypothesis, we can define an \emph{O'Neill-type operator} $A^{\mathcal{S}}: \mathfrak{X}(\mathcal{H}) \times \mathfrak{X}(\mathcal{H}) \to \mathfrak{X}(\mathcal{S})$ by  
\[
A^{\mathcal{S}}(Z,W) = \frac{1}{2}[Z,W]^{\mathcal{S}}.
\]  
Then, $\mathcal{H}$ is integrable if and only if $A^{\mathcal{S}}$ is identically zero, independent of the choice of $\mathcal{S}$. Given $Z,W \in \mathfrak{X}_b(M)$, we can choose $\widetilde{[d\pi(Z),d\pi(W)]}$ such that $[Z,W] - \widetilde{[d\pi(Z),d\pi(W)]} \in \mathfrak{X}(\mathcal{S})$. Therefore,  
\[
\frac{1}{4}\mathbf{A}(Z,W) = g_p(A^{\mathcal{S}}(Z,W), A^{\mathcal{S}}(Z,W)).
\]
In particular, $\mathcal{H}$ is integrable if and only if $\mathbf{A} = 0$, since $\mathcal{S}$ is spacelike. 

For $X \in \mathfrak{X}(\mathcal{H}|_L)$, define the \emph{$\mathcal{S}$-shape operator} $S^{\mathcal{S}}_X: \mathcal{S} \to \mathcal{S}$ by
\[
S^{\mathcal{S}}_X(u) = -(\nabla_u X)^{\mathcal{S}}.
\]
The \emph{$\mathcal{S}$-mean curvature form} is the map $\kappa^{\mathcal{S}}: \mathfrak{X}(\mathcal{A}|_L) \to \mathcal{F}(L)$ defined by  
\[
\kappa^{\mathcal{S}}(X) = \operatorname{tr}(S_{X^{\mathcal{H}}}^{\mathcal{S}}) = -\sum_i g(\nabla_{E_i}X^{\mathcal{H}}, E_i),
\]
where $\{E_i\}$ is a local orthonormal frame of $\mathcal{S}$ with respect to the decomposition $\mathcal{A} = \mathcal{S} \oplus \mathcal{H}$.  

\begin{proposition}
If $\mathcal{A}$ is an involutive distribution and $X, Y \in \mathfrak{X}_b(M)$, then 
\[
d\kappa^{\mathcal{S}}(X,Y) = -2\sum_i g(T_i, \nabla_{T_i} A^{\mathcal{S}}(X,Y)),
\]
where $\{T_i\}_{i=1}^{n}$ is a local orthonormal frame of $\mathcal{S}$.
\end{proposition}
\begin{proof}
Since $\kappa^{\mathcal{S}}([X,Y]^{\mathcal{S}}) = 0$, it follows that $\kappa^{\mathcal{S}}([X,Y]^{\mathcal{H}}) = \kappa^{\mathcal{S}}([X,Y])$. We also have the standard identity
\[
d\kappa^{\mathcal{S}}(X,Y) = X(\kappa^{\mathcal{S}}(Y)) - Y(\kappa^{\mathcal{S}}(X)) - \kappa^{\mathcal{S}}([X,Y]).
\]
Hence, it suffices to verify that
\begin{equation}\label{eqKXY}
    \kappa^{\mathcal{S}}([X,Y]^{\mathcal{H}}) = X(\kappa^{\mathcal{S}}(Y)) - Y(\kappa^{\mathcal{S}}(X)) + 2\sum_i g(T_i, \nabla_{T_i} A^{\mathcal{S}}(X,Y)).
\end{equation}

Let $\{T_i\}_{i=1}^{n}$ be a local orthonormal frame of $\mathcal{S}$. Define $X_{ij} = g([X,T_i],T_j)$ and $Y_{ij} = g([Y,T_i],T_j)$. Since $X$ and $Y$ are basic vector fields, $[X,T_i]$ and $[Y,T_j]$ are vertical. Hence, there exist vector fields $V_1, V_2 \in \mathcal{R}$ such that
\[
[X,T_i] = \sum_{j} X_{ij}T_{j} + V_1, \qquad 
[Y,T_i] = \sum_{j} Y_{ij}T_{j} + V_2.
\]

Note that
\begin{align*}  
 \kappa^{\mathcal{S}}([X,Y]^{\mathcal{H}})  &= \sum_{i} g([[X,Y]^{\mathcal{H}}, T_i], T_i)  \\
&= \sum_{i} \big( g([[X,Y], T_i], T_i) - g([[X,Y]^{\mathcal{S}}, T_i], T_i) \big) \\
&= \sum_{i} g([[X,Y], T_i], T_i) + 2\sum_{i} g(T_i, \nabla_{T_i} A^{\mathcal{S}}(X,Y)).                         
\end{align*}

By the Jacobi identity, we have
\begin{align*}  
\sum_{i} g([[X,Y], T_i], T_i)  
&= \sum_{i} \big\{ g([[X,T_i], Y], T_i) - g([[Y,T_i], X], T_i) \big\} \\
&= \sum_{i,j} \big\{ g([X_{ij}T_j, Y], T_i) + g([V_1, Y], T_i) - g([Y_{ij}T_j, X], T_i) - g([V_2, X], T_i) \big\} \\
&= \sum_{i,j} \big\{ g(X_{ij}[T_{j}, Y] - Y(X_{ij})T_j, T_i) - g(Y_{ij}[T_j, X] - X(Y_{ij})T_j, T_i) \big\} \\
&= \sum_{i,j} \big\{ X_{ij} g([T_j, Y], T_i) - Y(X_{ij}) g(T_j, T_i) \big\} \\
&\quad - \sum_{i,j} \big\{ Y_{ij} g([T_j, X], T_i) - X(Y_{ij}) g(T_j, T_i) \big\} \\
&= \sum_i \big\{ X(Y_{ii}) - Y(X_{ii}) \big\} \\
&= X(\kappa^{\mathcal{S}}(Y)) - Y(\kappa^{\mathcal{S}}(X)).
\end{align*}

Thus, equation \eqref{eqKXY} is verified.
\end{proof}

Let $\pi: (M,g) \to (B,\mathcal{D},h)$ be a generalized semi-Riemannian submersion with an involutive distribution $\mathcal{A}$ (in particular, $\mathcal{D}$ is involutive). Since $\mathcal{H}$ is locally finitely generated by basic vector fields, $\mathcal{D}$ is locally finitely generated and thus integrable by Sussmann's theorem (Theorem~\ref{Tstefansussman}).

Let $\mathcal{F}_{\mathcal{D}}$ be the integral foliation of $\mathcal{D}$. By construction, $(L_x, h^x)$ is a semi-Riemannian manifold for each leaf $L_x \in \mathcal{F}_{\mathcal{D}}$. Let $\nabla^h$ denote the map  
\[
\nabla^h: \mathfrak{X}(\mathcal{D}) \times \mathfrak{X}(\mathcal{D}) \to \mathfrak{X}(\mathcal{D})
\]
such that $(\nabla^h_X Y)|_{L_x}$ coincides with the Levi-Civita connection of $(L_x, h^x)$. In particular, it satisfies  
\[
Z h(X,Y) = h(\nabla^h_Z X,Y) + h(X,\nabla^h_Z Y),  
\]
\[
[X,Y] = \nabla^h_X Y - \nabla^h_Y X.
\]
It is uniquely determined by the Koszul formula:
\begin{multline}
    \label{eq:koszul_base}
2h(\nabla^h_X Y,Z) = X h(Y,Z) + Y h(X,Z) - Z h(X,Y) \\ - h([Y,Z],X) - h([X,Z],Y) + h([X,Y],Z).
\end{multline}
We refer to $\nabla^h$ as the Levi-Civita connection of the generalized semi-Riemannian metric $(\mathcal{D}, h)$.  

\begin{proposition}\label{prop:connection_lift}  
Let $\pi: (M,g) \to (B,\mathcal{D},h)$ be a generalized semi-Riemannian submersion such that $\mathcal{D}$ is integrable, and let $\nabla^h$ be the Levi-Civita connection of $h$. Given $X,Y \in \mathfrak{X}(\mathcal{D})$, let $\tilde{X}, \tilde{Y}$, and $\widetilde{\nabla^h_X Y}$ be the horizontal lifts of $X, Y$, and $\nabla^h_X Y$, respectively. Then,  
\[
\nabla_{\tilde{X}} \tilde{Y} - \widetilde{\nabla^h_X Y}
\]
is vertical. In particular, $\nabla_{\tilde{X}} \tilde{Y} \in \mathfrak{X}(\mathcal{A})$ and $d\pi(\nabla_{\tilde{X}} \tilde{Y}) = \nabla^h_X Y$.  
\end{proposition}  
\begin{proof}  
Given $Z \in \mathfrak{X}(\mathcal{D})$, let $\tilde{Z}$ be its horizontal lift. Using the Koszul formula and Lemma~\ref{Lema1}, we obtain  
\begin{align*}
2 g(\nabla_{\tilde{X}} \tilde{Y}, \tilde{Z}) &= \tilde{X} g(\tilde{Y},\tilde{Z}) + \tilde{Y} g(\tilde{X},\tilde{Z}) - \tilde{Z} g(\tilde{X},\tilde{Y})   \\
&\quad + g([\tilde{X},\tilde{Y}],\tilde{Z}) - g([\tilde{Y},\tilde{Z}],\tilde{X}) + g([\tilde{Z},\tilde{X}],\tilde{Y}) \\
&= X h(Y,Z) + Y h(X,Z) - Z h(X,Y) \\
&\quad + h([X,Y],Z) - h([Y,Z],X) + h([Z,X],Y) \\
&= 2h(\nabla^h_X Y, Z) = 2g(\widetilde{\nabla^h_X Y}, \tilde{Z}).
\end{align*}
Thus, $g(\nabla_{\tilde{X}} \tilde{Y} - \widetilde{\nabla^h_X Y}, \tilde{Z}) = 0$, which implies that $\nabla_{\tilde{X}} \tilde{Y} - \widetilde{\nabla^h_X Y}$ is orthogonal to $\mathcal{H}$, and therefore vertical.  
\end{proof}  

\begin{proposition}\label{propimportante} 
If $\pi: (M,g) \to (B,\mathcal{D},h)$ is a generalized semi-Riemannian submersion such that $\mathcal{D}$ is integrable, then there exists $V_0 \in \mathfrak{X}(\mathcal{R})$ such that   
\[
\nabla_{\tilde{X}} \tilde{Y} = \widetilde{\nabla^h_X Y} + \frac{1}{2} ([\tilde{X},\tilde{Y}] - \widetilde{[X,Y]}) + V_0
\]
for all $X,Y \in \mathfrak{X}(\mathcal{D})$, where $\tilde{X}, \tilde{Y}$, and $\widetilde{\nabla^h_X Y}$ are their respective horizontal lifts.
\end{proposition}  
\begin{proof}  
Let $X, Y \in \mathfrak{X}(\mathcal{D})$ and $U \in \mathfrak{X}(\mathcal{V})$. By the Koszul formula, we obtain  
\[
2g(\nabla_{\tilde{X}} \tilde{Y}, U) = g([\tilde{X},\tilde{Y}], U).
\]
Using the properties of basic vector fields and verticality, it follows that  
\[
g(\nabla_{\tilde{X}} \tilde{Y} - \widetilde{\nabla^h_X Y}, U) = \frac{1}{2} g(V, U),
\]
where $V := [\tilde{X},\tilde{Y}] - \widetilde{[X,Y]} \in \mathfrak{X}(\mathcal{V})$. Since $\nabla_{\tilde{X}} \tilde{Y} - \widetilde{\nabla^h_X Y}$ is vertical (Proposition~\ref{prop:connection_lift}), applying Lemma~\ref{Pontochave} yields the existence of $V_0 \in \mathfrak{X}(\mathcal{R})$ such that  
\[
\nabla_{\tilde{X}} \tilde{Y} - \widetilde{\nabla^h_X Y} = \frac{1}{2}V + V_0.
\]
\end{proof}  

\begin{theorem}[O'Neill's Formula]\label{FormuladeOneill} 
Let $\pi: (M,g) \to (B,\mathcal{D},h)$ be a generalized semi-Riemannian submersion such that $\mathcal{D}$ is an involutive distribution. Then
\[
K^h_{\pi(p)}(d\pi(X),d\pi(Y)) = K^g_p(X,Y) + \frac{3}{4}\mathbf{A}(X,Y)
\]
for each $p \in M$ and $X,Y \in \mathfrak{X}_b(M)$ such that $X(p),Y(p)$ are orthonormal vectors, where $K^h_{\pi(p)}$ and $K^g_{p}$ denote the sectional curvatures of $(L_{\pi(p)},h) \in \mathcal{F}_{\mathcal{D}}$ and $(M,g)$, respectively.
\end{theorem}
\begin{proof} 
We begin by verifying the following claim: if $U \in \mathfrak{X}(\mathcal{V})$, then  
\begin{equation}\label{EqClaim}
g(\nabla_U \tilde{X}, \tilde{Y}) = g(\nabla_{\tilde{X}} U, \tilde{Y}) = -\frac{1}{2} g(U,V),
\end{equation}
where $V := [\tilde{X},\tilde{Y}] - \widetilde{[X,Y]} \in \mathfrak{X}(\mathcal{V})$ (Lemma~\ref{Lema1}). In particular, setting $U = V$ yields  
\begin{equation}\label{EqInicial}
g(\nabla_V \tilde{X}, \tilde{Y}) = g(\nabla_{\tilde{X}} V, \tilde{Y}) = -\frac{1}{2} g(V,V).
\end{equation}

To prove \eqref{EqClaim}, recall from Lemma~\ref{Lema1} that  
\[
g(\nabla_U \tilde{X}, \tilde{Y}) = g(\nabla_{\tilde{X}} U, \tilde{Y}) + g([U, \tilde{X}], \tilde{Y}) = g(\nabla_{\tilde{X}} U, \tilde{Y}).
\]
Using metric compatibility, Proposition~\ref{propimportante}, and Lemma~\ref{Pontochave}, we obtain  
\[
g(\nabla_{\tilde{X}} U, \tilde{Y}) = -g(U, \nabla_{\tilde{X}} \tilde{Y}) = -g\left(U, \widetilde{\nabla^h_X Y} + \frac{1}{2}V + V_0\right) = -\frac{1}{2} g(U,V),
\]
which proves claim \eqref{EqClaim}.

Now, observe that $[\tilde{Y},\tilde{X}] - \widetilde{[Y,X]} = -V$. Thus, by Proposition~\ref{propimportante}, there exist $V_0, V_1 \in \mathfrak{X}(\mathcal{R})$ such that  
\begin{align*}
\nabla_{\tilde{X}} \nabla_{\tilde{Y}} \tilde{X} &= \nabla_{\tilde{X}} \widetilde{\nabla^h_Y X} - \frac{1}{2} \nabla_{\tilde{X}} V + \nabla_{\tilde{X}} V_0 \\
&= \widetilde{\nabla^h_X \nabla^h_Y X} + \frac{1}{2} \hat{V} + V_1 - \frac{1}{2} \nabla_{\tilde{X}} V + \nabla_{\tilde{X}} V_0,
\end{align*}
where $\hat{V} := [\tilde{X}, \widetilde{\nabla^h_Y X}] - \widetilde{[X, \nabla^h_Y X]} \in \mathfrak{X}(\mathcal{V})$.

On the other hand, since $\nabla_{\tilde{X}} \tilde{Y} \in \mathfrak{X}(\mathcal{A})$ (Proposition~\ref{prop:connection_lift}), metric compatibility and Lemma~\ref{Pontochave} imply  
\[
g(\nabla_{\tilde{X}} V_0, \tilde{Y}) = -g(V_0, \nabla_{\tilde{X}} \tilde{Y}) = 0 = g(V_1, \tilde{Y}).
\]
Thus, using claim \eqref{EqInicial} and the fact that $\hat{V} \in \mathfrak{X}(\mathcal{V})$, we obtain  
\begin{align*}
g(\nabla_{\tilde{X}} \nabla_{\tilde{Y}} \tilde{X}, \tilde{Y})
&= g(\widetilde{\nabla^h_X \nabla^h_Y X}, \tilde{Y}) + \frac{1}{2} g(\hat{V}, \tilde{Y}) + g(V_1, \tilde{Y}) - \frac{1}{2} g(\nabla_{\tilde{X}} V, \tilde{Y}) + g(\nabla_{\tilde{X}} V_0, \tilde{Y}) \\
&= g(\widetilde{\nabla^h_X \nabla^h_Y X}, \tilde{Y}) - \frac{1}{2} g(\nabla_{\tilde{X}} V, \tilde{Y}) \\
&= g(\widetilde{\nabla^h_X \nabla^h_Y X}, \tilde{Y}) + \frac{1}{4} g(V,V) \\
&= h(\nabla^h_X \nabla^h_Y X, Y) + \frac{1}{4} g(V,V).
\end{align*}

Similarly, we compute  
\[
g(\nabla_{\tilde{Y}} \nabla_{\tilde{X}} \tilde{X}, \tilde{Y}) = g(\widetilde{\nabla^h_Y \nabla^h_X X}, \tilde{Y}) = h(\nabla^h_Y \nabla^h_X X, Y),
\]
and using claim \eqref{EqInicial} alongside Proposition~\ref{propimportante}, we get  
\begin{align*}
g(\nabla_{[\tilde{X},\tilde{Y}]} \tilde{X}, \tilde{Y}) &= g(\nabla_{\widetilde{[X,Y]}} \tilde{X}, \tilde{Y}) + g(\nabla_V \tilde{X}, \tilde{Y}) \\ 
&= h(\nabla^h_{[X,Y]} X, Y) - \frac{1}{2} g(V,V).
\end{align*}

Therefore, combining the previous expressions, we arrive at  
\begin{align*}
g(R^g(\tilde{X},\tilde{Y})\tilde{X}, \tilde{Y}) &= g(\nabla_{\tilde{X}} \nabla_{\tilde{Y}} \tilde{X} - \nabla_{\tilde{Y}} \nabla_{\tilde{X}} \tilde{X} - \nabla_{[\tilde{X},\tilde{Y}]} \tilde{X}, \tilde{Y}) \\
&= h(\nabla^h_X \nabla^h_Y X - \nabla^h_Y \nabla^h_X X - \nabla^h_{[X,Y]} X, Y) + \frac{3}{4} g(V,V) \\
&= h(R^h(X,Y)X, Y) + \frac{3}{4} g(V,V).
\end{align*}
Using the standard sign convention $K(X,Y) = g(R(X,Y)Y, X) = -g(R(X,Y)X, Y)$ for orthonormal vectors, this yields  
\[
K^g(\tilde{X},\tilde{Y}) = K^h(X,Y) - \frac{3}{4} \mathbf{A}(X,Y),
\]
which rearranges directly to $K^h(X,Y) = K^g(\tilde{X},\tilde{Y}) + \frac{3}{4}\mathbf{A}(X,Y)$.
\end{proof}  

\section{Appendix: A Fact on Scalar Products}

We assume familiarity with semi-Riemannian geometry (see \cite{o1983semi} as a general reference). For convenience, we present the following elementary result, which plays a central role in our work.  

\begin{lemma}\label{Pontochave}  
Let $(W,g)$ be a vector space equipped with a scalar product, let $V \subset W$ be a subspace, $H = V^{\perp}$, $R = V \cap H$, and $A = V + H$. Then:  
\begin{enumerate}  
    \item $v \in R$ if and only if $g(v,u) = 0$ for all $u \in A$;  
    \item if $v_1,v_2 \in R$ and $u_1, u_2 \in A$, then $g(u_1 + v_1, u_2 + v_2) = g(u_1, u_2)$;  
    \item if $u_1, u_2 \in W$ satisfy $g(u_1, w) = g(u_2, w)$ for all $w \in A$, then there exists $v \in R$ such that $u_1 = u_2 + v$;  
    \item if $D$ is a subspace of $V$ (or $H$) such that $H = D \oplus R$ or $V = D \oplus R$, then $D$ is non-degenerate.  
\end{enumerate}  
\end{lemma}  
\begin{proof}  
Given $u \in A$, write $u = u_1 + u_2$ with $u_1 \in H$ and $u_2 \in V$. If $v \in R = V \cap H$, then $g(v, u_1) = 0$ and $g(v, u_2) = 0$, so $g(u, v) = 0$ by linearity. Conversely, if $g(v,u) = 0$ for all $u \in A$, then $v \in A^\perp = (V + H)^\perp = V^\perp \cap H^\perp = H \cap V = R$, completing the proof of Item 1. Item 2 follows directly from Item 1. Item 3 holds because $g(u_1 - u_2, w) = 0$ for all $w \in A$, which by Item 1 implies $u_1 - u_2 \in R$.

Finally, we prove Item 4. Suppose $H = D \oplus R$. Let $w \in D$ satisfy $g(w, v) = 0$ for all $v \in D$. Given any $u \in H$, write $u = v + x$ with $v \in D$ and $x \in R$. Then $g(w, u) = g(w, v) + g(w, x) = 0$, implying $w \in H^{\perp} = V$. Thus $w \in D \cap V \subset H \cap V = R$, which implies $w \in D \cap R = \{0\}$. An identical argument holds if $V = D \oplus R$, proving that $D$ is non-degenerate in both cases.  
\end{proof}


\begin{thebibliography}{31}

\bibitem{aazami2023finsler}
Amir Babak Aazami, Miguel \'Angel Javaloyes, and Marcus~C Werner.
\newblock Finsler pp-waves and the Penrose limit.
\newblock {\em General Relativity and Gravitation}, 55(3):52, 2023.

\bibitem{agrachev2013control}
Andrei~A Agrachev and Yuri Sachkov.
\newblock {\em Control theory from the geometric viewpoint}, volume~87.
\newblock Springer Science \& Business Media, 2013.

\bibitem{alexandrino2024equifocal}
Marcos~M Alexandrino, Benigno Alves, and Miguel~Angel Javaloyes.
\newblock On equifocal Finsler submanifolds and analytic maps.
\newblock {\em Israel Journal of Mathematics}, 259(1):203--237, 2024.

\bibitem{alexandrino2019finsler}
Marcos~M Alexandrino, Benigno~O Alves, and Hengameh~R Dehkordi.
\newblock On Finsler transnormal functions.
\newblock {\em Differential Geometry and its Applications}, 65:93--107, 2019.

\bibitem{alexandrino2019singular}
Marcos~M Alexandrino, Benigno~O Alves, and Miguel~Angel Javaloyes.
\newblock On singular Finsler foliation.
\newblock {\em Annali di Matematica Pura ed Applicata (1923-)}, 198(1):205--226, 2019.

\bibitem{alexandrino2015lie}
Marcos~M Alexandrino, Renato~G Bettiol, et~al.
\newblock {\em Lie groups and geometric aspects of isometric actions}, volume~82.
\newblock Springer, 2015.

\bibitem{alexandrino2013progress}
Marcos~M Alexandrino, Rafael Briquet, and Dirk T{\"o}ben.
\newblock Progress in the theory of singular Riemannian foliations.
\newblock {\em Differential Geometry and its Applications}, 31(2):248--267, 2013.

\bibitem{alexandrino2022lie}
Marcos~M Alexandrino, Marcelo~K Inagaki, Mateus de~Melo, and Ivan Struchiner.
\newblock Lie groupoids and semi-local models of singular Riemannian foliations.
\newblock {\em Annals of Global Analysis and Geometry}, 61(3):593--619, 2022.

\bibitem{allout2022homogeneous}
Souheib Allout, Abderrahmane Belkacem, and Abdelghani Zeghib.
\newblock On homogeneous 3-dimensional spacetimes: focus on plane waves.
\newblock {\em Geometriae Dedicata}, 2025.

\bibitem{alves2024isoparametric}
Benigno~Oliveira Alves and Patr{\'\i}cia Mar{\c{c}}al.
\newblock Isoparametric functions and mean curvature in manifolds with Zermelo navigation.
\newblock {\em Annali di Matematica Pura ed Applicata (1923-)}, 203(3):1285--1310, 2024.

\bibitem{bolosgeometric}
Vicente~J Bol{\'o}s.
\newblock Geometric description of lightlike foliations by an observer in general relativity.
\newblock In {\em Mathematical Proceedings of the Cambridge Philosophical Society}, volume 139, pages 181--192. Cambridge University Press, 2005.

\bibitem{caramello2024transverse}
Francisco~C Caramello~Jr, Henrique~A~Puel Martins, and Ivan~P Costa~e~Silva.
\newblock Transverse geometry of Lorentzian foliations with applications to Lorentzian orbifolds.
\newblock {\em arXiv preprint arXiv:2402.05907}, 2024.

\bibitem{cartan1938familles}
{\'E}lie Cartan.
\newblock Familles de surfaces isoparam{\'e}triques dans les espaces {\`a} courbure constante.
\newblock {\em Annali di Matematica Pura ed Applicata}, 17(1):177--191, 1938.

\bibitem{chaib2025isometries}
Salah Chaib, Ana~Cristina Ferreira, and Abdelghani Zeghib.
\newblock Isometries of 3-dimensional semi-Riemannian Lie groups.
\newblock {\em Transformation Groups}, 2025.

\bibitem{chaves2020foliations}
Rosa Maria dos Santos~Barreiro Chaves and Euripedes Carvalho~da Silva.
\newblock Foliations by spacelike hypersurfaces on Lorentz manifolds.
\newblock {\em Results in Mathematics}, 75(1):36, 2020.

\bibitem{dehkordi2019huygens}
Hengameh~R Dehkordi and Alberto Saa.
\newblock Huygens' envelope principle in Finsler spaces and analogue gravity.
\newblock {\em Classical and Quantum Gravity}, 36(8):085008, 2019.

\bibitem{dolgonosova2018pseudo}
A~Yu Dolgonosova and NI~Zhukova.
\newblock Pseudo-Riemannian foliations and their graphs.
\newblock {\em Lobachevskii Journal of Mathematics}, 39:54--64, 2018.

\bibitem{duggal2012foliations}
Krishan Duggal.
\newblock Foliations of lightlike hypersurfaces and their physical interpretation.
\newblock {\em Central European Journal of Mathematics}, 10(5):1789--1800, 2012.

\bibitem{galloway2004null}
Gregory~J Galloway.
\newblock Null geometry and the Einstein equations.
\newblock In {\em The Einstein Equations and the Large Scale Behavior of Gravitational Fields: 50 Years of the Cauchy Problem in General Relativity}, pages 379--400. Springer, 2004.

\bibitem{huber2023fundamental}
Matthieu Huber and Miguel~Angel Javaloyes.
\newblock The fundamental equations of a pseudo-Finsler submersion.
\newblock {\em Annali di Matematica Pura ed Applicata (1923-)}, 202(4):1877--1905, 2023.

\bibitem{javaloyes2023general}
Miguel~A Javaloyes, Enrique Pend{\'a}s-Recondo, and Miguel S{\'a}nchez.
\newblock A general model for wildfire propagation with wind and slope.
\newblock {\em SIAM Journal on Applied Algebra and Geometry}, 7(2):414--439, 2023.

\bibitem{lee2013smooth}
John~M Lee.
\newblock Smooth manifolds.
\newblock In {\em Introduction to smooth manifolds}, pages 1--31. Springer, 2013.

\bibitem{lewis1998affine}
Andrew~D Lewis.
\newblock Affine connections and distributions with applications to nonholonomic mechanics.
\newblock {\em Reports on Mathematical Physics}, 42(1-2):135--164, 1998.

\bibitem{lytchak2016riemannian}
Alexander Lytchak and Burkhard Wilking.
\newblock Riemannian foliations of spheres.
\newblock {\em Geometry \& Topology}, 20(3):1257--1274, 2016.

\bibitem{montiel1999uniqueness}
Sebasti{\'a}n Montiel.
\newblock Uniqueness of spacelike hypersurfaces of constant mean curvature in foliated spacetimes.
\newblock {\em Mathematische Annalen}, 314(3):529--553, 1999.

\bibitem{o1966fundamental}
Barrett O'Neill.
\newblock The fundamental equations of a submersion.
\newblock {\em Michigan Mathematical Journal}, 13(4):459--469, 1966.

\bibitem{o1983semi}
Barrett O'Neill.
\newblock {\em Semi-Riemannian geometry with applications to relativity}.
\newblock Academic Press, 1983.

\bibitem{qadir2006foliation}
Asghar Qadir and Azad~A Siddiqui.
\newblock Foliation of the Schwarzschild and Reissner--Nordstr{\"o}m space--times by flat space-like hypersurfaces.
\newblock {\em International Journal of Modern Physics D}, 15(09):1419--1440, 2006.

\bibitem{walschap1999spacelike}
Gerard Walschap.
\newblock Spacelike metric foliations.
\newblock {\em Journal of Geometry and Physics}, 32(2):97--101, 1999.

\bibitem{wang1987isoparametric}
Qi-Ming Wang.
\newblock Isoparametric functions on Riemannian manifolds. {I}.
\newblock {\em Mathematische Annalen}, 277(4):639--646, 1987.

\bibitem{zeghib1999isometry}
Abdelghani Zeghib.
\newblock Isometry groups and geodesic foliations of Lorentz manifolds. {Part I}: Foundations of Lorentz dynamics.
\newblock {\em Geometric \& Functional Analysis GAFA}, 9(4):775--822, 1999.

\end{thebibliography}
\end{document}